\documentclass[11pt]{amsart}
\usepackage{stmaryrd}
\usepackage{mathrsfs}
\usepackage{amssymb}
\usepackage{color}
\usepackage{titletoc}
\input xy
  \xyoption{all}
\usepackage{amscd}
\usepackage{amsmath, amssymb}
\usepackage{amsfonts}
\usepackage[colorlinks,linkcolor=blue,citecolor=blue, pdfstartview=FitH]{hyperref}

\numberwithin{equation}{section}

\theoremstyle{plain}
\newtheorem{thm}{Theorem}[section]
\newtheorem{theorem}[thm]{Theorem}
\newtheorem{lemma}[thm]{Lemma}
\newtheorem{corollary}[thm]{Corollary}
\newtheorem{proposition}[thm]{Proposition}

\theoremstyle{definition}

\newtheorem{remark}[thm]{Remark}

\newtheorem{definition}[thm]{Definition}

\newtheorem{example}[thm]{Example}

\newtheorem{defn-thm}[thm]{Definition-Theorem}
\newtheorem{conjecture}[thm]{Conjecture}

\theoremstyle{remark}
\newtheorem{rem}[thm]{Remark}

\newcommand{\sO}{{\mathcal O}}

\renewcommand{\P}{{\mathbb P}}
\newcommand{\Q}{{\mathbb Q}}

\newcommand{\Sym}{{ Sym}}

\newcommand{\ds}{\oplus}

\newcommand{\ts}{\otimes}

\newcommand{\btheorem}{\begin{theorem}}
\newcommand{\etheorem}{\end{theorem}}
\newcommand{\bproposition}{\begin{proposition}}
\newcommand{\eproposition}{\end{proposition}}
\newcommand{\bdefinition}{\begin{definition}}
\newcommand{\edefinition}{\end{definition}}
\newcommand{\bcorollary}{\begin{corollary}}
\newcommand{\ecorollary}{\end{corollary}}
\newcommand{\bproof}{\begin{proof}}
\newcommand{\eproof}{\end{proof}}
\newcommand{\bremark}{\begin{remark}}
\newcommand{\eremark}{\end{remark}}
\newcommand{\eexample}{\end{example}}
\newcommand{\bexample}{\begin{example}}

\newcommand{\elemma}{\end{lemma}}
\newcommand{\blemma}{\begin{lemma}}

\renewcommand{\phi}{\varphi}

\newcommand{\ee}{\end{eqnarray*}}
\newcommand{\be}{\begin{eqnarray*}}

\newcommand{\beq}{\begin{equation}}
\newcommand{\eeq}{\end{equation}}

\newcommand{\bd}{\begin{enumerate}}
\newcommand{\ed}{\end{enumerate}}

\renewcommand{\tilde}{\widetilde}

\renewcommand{\>}{\rightarrow}

\usepackage{fancyhdr}
\renewcommand{\leftmark}{{\scriptsize On projective varieties with strictly nef tangent bundles}}

\begin{document}
\title{{\color{blue}{On projective varieties with strictly nef tangent bundles}}}
\makeatletter
\let\uppercasenonmath\@gobble
\let\MakeUppercase\relax
\let\scshape\relax
\makeatother
\author{ Duo Li, Wenhao Ou and Xiaokui Yang}

\address{{Duo Li, Yau Mathematical Sciences Center, Tsinghua University, Beijing, 100084, P. R. China.}}
\email{\href{mailto:duoli@math.tsinghua.edu.cn}{{duoli@math.tsinghua.edu.cn}}}

\address{Wenhao Ou, UCLA Mathematics Department,
520 Portola Plaza, Los Angeles, CA 90095, USA}
\email{\href{mailto:wenhaoou@math.ucla.edu}{wenhaoou@math.ucla.edu}}

\address{{Xiaokui Yang, Morningside Center of Mathematics, Institute of
        Mathematics, HCMS, CEMS, NCNIS, HLM, UCAS, Academy of Mathematics and Systems Science, Chinese Academy of Sciences, Beijing 100190, China.}}
\email{\href{mailto:xkyang@amss.ac.cn}{{xkyang@amss.ac.cn}}}
\maketitle

\begin{abstract}
In this paper, we study smooth complex projective varieties $X$ such that some exterior power $\bigwedge^r T_X$ of the tangent bundle is strictly nef.  We prove that such varieties are
rationally connected. We also classify the following two cases. If  $T_X$ is strictly nef, then $X$ isomorphic to the projective space $\P^n$. If $\bigwedge^2 T_X$ is strictly nef and if $X$ has dimension at least $3$, then $X$ is either isomorphic to $\P^n$ or a  quadric $\Q^n$.
\end{abstract}

\setcounter{tocdepth}{1} \tableofcontents

\section{Introduction}

\noindent  Throughout this paper, we will study projective varieties defined over $\mathbb{C}$, the field of complex numbers.   We recall that a  line
bundle $L$ over a smooth projective variety $X$ is said to be \emph{strictly nef} if there is  $$L\cdot
C>0$$ for any  curve $C\subset X$, while the
Nakai-Moishezon-Kleiman criterion asserts that  $L$ is
\emph{ample} if and only if there is $$ L^{\dim Y}\cdot Y>0 $$ for every
positive-dimensional  subvariety $Y\subset X$. In particular, ampleness implies strict nefness. However, the converse is not true in general, as shown in an example of Mumford (see \cite[Section 10, Chapter I]{Har70}).  Nevertheless, one might expect more for the canonical bundle $\omega_X$ of $X$. Indeed, on the one hand, since a strictly nef  semi-ample line bundle must be ample,  the abundance conjecture  suggests that if $\omega_X$ is strictly nef, then it should be ample.  On the other hand, Campana and Peternell proposed in
\cite[Problem ~11.4]{CP91} the following conjecture.

\begin{conjecture} \label{c} Let $X$ be a smooth   projective
variety. If $\omega_X^{-1}$ is strictly nef, then $\omega_X^{-1}$ is ample, that is, $X$
is a Fano variety.
\end{conjecture}

\noindent This conjecture is only verified by Maeda for surfaces
(see \cite{Mae93}) and by Serrano  for threefolds (see \cite{Ser95},
also \cite{Ueh00} and \cite{CCP06}). In this paper, we prove the following theorem, which provides some evidence for this conjecture in all dimensions.

\btheorem\label{main11} Let $X$ be a smooth  projective variety of
dimension $n$, and let $T_X$ be its tangent bundle. If $\bigwedge^r T_X$ is strictly nef for some
$1\leqslant r \leqslant n$, then $X$ is rationally connected.  In
particular, if $\omega_X^{-1}$ is strictly nef, then $X$ is rationally
connected. \etheorem

\noindent We recall that a vector bundle $E$ is said to be strictly nef
  if the tautological line bundle $\sO_E(1)$ on the
 projective bundle $\mathrm{Proj}\, E$ of hyperplanes is strictly nef.

  One of the key ingredients for the proof of Theorem
\ref{main11} relies on the recent breakthrough of   Cao and  H\"{o}ring on the structure
theorems for projective varieties with nef anticanonical bundle (see \cite{Cao16} and  \cite{CH17}).
Actually, based on their work, we prove the following result, which is   essential for Theorem \ref{main11}.

\begin{thm}
\label{thm-section2} Let $X$ be a smooth projective  variety with nef anticanonical bundle $\omega_X^{-1}$. Then, up to replacing $X$ with some finite \'etale cover if necessary,   the Albanese morphism $f:X\to A$ has a section $\sigma:A\to X$ such that $\sigma^*\omega_X$ is numerically trivial.
\end{thm}

\noindent  As an application of Theorem \ref{main11}, we prove the following  analogue of Mori's  characterization of projective spaces (see \cite{Mor79}).
 \btheorem\label{main2} Let
$X$ be a smooth projective variety of dimension $n$. If $T_X$ is
 strictly nef, then  $X$ is $ \P^n$. \etheorem

\noindent In another word,  this theorem states that the tangent bundle    $T_X$ is
strictly nef if and only if  it is ample. Along the same lines as
Theorem \ref{main2}, we obtain the following characterization  as well.

\btheorem\label{main3} Let $X$ be a smooth projective variety of
 dimension $n\geq 3$. Suppose that $\bigwedge^2T_X$ is strictly nef,
then $X$ is isomorphic to $\P^n$ or a quadric $\Q^n$.
 \etheorem

\noindent There are also  other   characterizations of  projective spaces or quadrics,  see, for example,
\cite{KO73}, \cite{SY80}, \cite{Siu80}, \cite{Mok88}, \cite{YZ90},
\cite{Pet90, Pet91}, \cite{CS95}, \cite{AW01}, \cite{CMSB02},  \cite{Miy04}, \cite{ADK08},
\cite{Hwa13}, \cite{DH} , \cite{Yang17}, etc.
 \vskip 0.4\baselineskip

\noindent \textbf{Acknowledgements.}  The authors  would like to thank  Shing-Tung Yau for his comments and suggestions on an earlier version of this paper
which clarify and improve the presentations. The authors would also
like to thank Yifei Chen, Yi Gu  and   Xiaotao Sun  for some useful
discussions.
The first author is very grateful to  Baohua Fu for his support, encouragement and stimulating discussions over the last few years.
The second author would like to thank Meng Chen, Zhi Jiang and Burt Totaro  for fruitful conversations.
The third author wishes to thank  Kefeng Liu, Valentino Tosatti and Xiangyu Zhou for  helpful discussions.
They would also like to thank the referee for pointing out a more general version of  Lemma \ref{flat}.
This work was partially supported by China's Recruitment
 Program of Global Experts and  NSFC 11688101.

\vskip 2\baselineskip

\section{Basic properties of   strictly nef vector
bundles}\label{basic}

\noindent In this section, we collect some basic results on  strictly nef vector bundles. The following proposition
is analogous to the Barton-Kleiman criterion for nef vector bundles
(see for example \cite[Proposition~6.1.18]{Laz04II}).

\bproposition\label{cri0} Let $E$ be a vector bundle on a smooth projective variety $X$. Then the following conditions are equivalent.
\bd
\item $E$ is  strictly nef.

\item For any    smooth projective curve $C$ with a  finite morphism $\nu:C\>X$,  and for any line bundle quotient $\nu^*(E)\to L$, one has $$ \deg \, L>0.  $$ \ed

\eproposition

\bproof

For a fixed smooth projective curve $C$, we know that  any non-constant morphism $\mu:C\to \mathrm{Proj}\, E$ whose image is horizontal over $X$ corresponds one-to-one to a finite morphism $\nu:C\to X$ with a  line bundle quotient $\nu^*E\to  L$. Moreover, we  have the following  commutative diagram

\beq \xymatrix{
  C \ar[rr]^{\mu} \ar[dr]_{\nu}
                &  &    \mathrm{Proj}\, E \ar[dl]^{}    \\
                & X                }
\label{diag}\eeq
such that  $L=\mu^* \sO_E(1) $, where  $\sO_E(1)$ is the tautological line bundle of $\mathrm{Proj}\, E$ (see for exmaple \cite[Proposition  II.7.12]{Har}).

$(1)\implies (2).$ Let $\nu:C\>X$ be a finite morphism, where $C$ is a smooth projective curve.  Assume that there is a line bundle quotient  $\nu^*E\to  L$.  Let $\mu:C\> \mathrm{Proj}\, E$ be the induced morphism as in diagram  (\ref{diag}). Then we have   $$ \deg\, L = \deg\, \mu^*\sO_E(1) = \sO_E(1) \cdot \mu(C).$$ Since  $\sO_E(1)$ is strictly nef, we obtain that $\deg\, L >0$.

$(2)\implies (1).$  Let $B$ be a  curve in $\mathrm{Proj}\, E$. Let  $f:  \widetilde{B} \to \mathrm{Proj}\, E$ be its normalization.  If $B$ is vertical over  $X$, then  we have $\sO_E(1)\cdot B>0$. If $B$ is horizontal over $X$, then the natural morphism  $g: \widetilde{B} \to X$ is  finite.  In this case,  there is an induced  line bundle quotient $g^*E\to Q$ such that $Q \cong f^*\sO_E(1)$. We note that $$\sO_E(1)\cdot B = \deg\, Q,$$ which is positive by hypothesis. Thus $E$ is strictly nef.  \eproof

\noindent We also have the following  list of  properties of strictly nef vector bundles.

\bproposition\label{directsum} Let $E$ and $F$ be two
vector bundles on a smooth projective variety $X.$ Then we have the following assertions.
\bd
\item $E$ is a  strictly nef if and only if for every smooth projective curve
$C$ and for any non-constant morphism $f: C\> X$,    $f^*E$ is strictly nef.

 \item If $E$ is  strictly nef, then any non-zero quotient bundle $Q$ of $E$ is  strictly nef.

\item  If $E\ds F$  is  strictly nef, then both $E$ and $F$ are
 strictly nef.

\item If the symmetric power $\Sym^k\, E$ is  strictly nef for some $k\geq 1$, then $E$ is  strictly nef.

\item Let  $f:Y\>X$ be  a  finite morphism such that  $Y$ is a  smooth projective variety. If $E$ is strictly nef, then so  is $f^*E$.

\item Let  $f:Y\>X$ be  a  surjective morphism such that  $Y$ is a  smooth projective variety. If $f^*E$ is strictly nef, then  $E$ is  strictly nef.

 \item If $E$ is strictly nef, then $h^0(X,E^*\ts L)=0$  for any numerically trivial line bundle $L$.
 \ed\eproposition

\begin{proof}
(1) follows directly from Proposition \ref{cri0}. For (2), we note that there is a natural embedding $\iota: \mathrm{Proj}\, Q \to \mathrm{Proj}\, E$ such that $\iota^*\sO_{E}(1) \cong \sO_{Q}(1)$. Hence if $E$ is strictly nef, then so is $Q$. (3) follows from $(2)$. For (4), we note that there is a Veronese embedding $v: \mathrm{Proj}\, E \to \mathrm{Proj}\, (\Sym^k\, E)$ such that  $v^*\sO_{\Sym^k\, E}(1) \cong \sO_{E}(k)$. This implies (4).

For (5), we notice that for any smooth projective curve $C$ with a finite morphism $\nu: C\to X$, the composition $f\circ \nu : C\> X$ is also finite. The assertion then follows from Proposition \ref{cri0}.

Now we consider (6). We note that for every curve in $X$, there is a curve in $Y$ which maps onto it. Hence by (1), we only need to prove the case when $X$ and $Y$ are smooth curves.  In this case, there is a natural finite surjective morphism $g:\mathrm{Proj}\, (f^*E) \to \mathrm{Proj}\, E$ induced by $f$. Moreover, we have $\sO_{f^*E}(1) \cong g^*\sO_E(1)$. By assumption, $\sO_{f^*E}(1)$ is a strictly nef line bundle. Since $g$ is finite surjective, this implies that $\sO_E(1)$ is also strictly nef. Therefore, $E$ is strictly nef.

It remains to prove (7). We remark that  $E\ts L^{-1}$ is still strictly nef for  $L^{-1}$ is numerically trivial. Thus, by replacing $E$ with $E\ts L^{-1}$, we may assume that $L$ is trivial. Assume by contradiction that $h^0(X, E^*)>0$. By \cite[Proposition~1.16]{DPS94},   there exists a nowhere vanishing section $\sigma\in H^0(X,  E^*)$. Then $\sigma$ induces a subbundle $\sO_X\> E^*$ as well as  a quotient bundle  $  E\>\sO_X.$ This contradicts Proposition \ref{cri0}.
\end{proof}

\vskip 1\baselineskip

\section{Strictly nef bundles on curves}\label{ellitpic}

\noindent In this section, we will look at  strictly nef vector bundle  $E$ on a smooth projective  curve $C$.  If $C$ is rational, then   $E$ is a direct sum of line bundles.   Hence $E$ is   strictly nef if and only if $E$ is ample in this case.  However, on a smooth curve $C$ of genus at least $2$, there exists a strictly nef vector bundle $E$ which is also a Hermitian flat stable vector bundle   (see \cite[Section 10 in Chapter I]{Har70}).  In particular, this  bundle $E$ is   not ample. Now it remains to look at the case when $C$ is elliptic. We observe the following fact.

\btheorem \label{main5} Let $E$ be a vector bundle on an elliptic
curve $C$. If  $E$ is strictly nef, then $E$ is ample. \etheorem

\noindent For the proof of this theorem, we will first prove the following lemma. Recall that a vector bundle $E$ is called numerically
flat  if both $E$ and $E^*$ are nef, or equivalently, if both $E$ and
$\det\, (E^*)$ are nef (see \cite[Definition~1.17]{DPS94}).

\blemma\label{flat} Let $E$  be a strictly nef vector bundle
on smooth projective variety $X$ whose Kodaira dimension $\kappa (X)$ satisfies $0 \leqslant \kappa (X)<\dim X$. Then  $\det\, E$ is not numerically trivial. \elemma

\bproof  We will first prove the case when the Kodaira dimension of $X$ is $0$.  Assume by contradiction that $\det \, E$ is numerically trivial. Then   $E$ is numerically flat, and so is $E^*$. By \cite[Theorem~1.18]{DPS94},  $E^*$ admits a filtration $$0=E_0\subset E_1\subset \cdots \subset E_p=E^*$$ of subbundles such that the quotients $E_k/E_{k-1}$ are Hermitian flat. In particular, $E_1$ is Hermitian flat, and is defined by a unitary representation of the fundamental group $\pi_1(X)$.  After \cite[Corollary 1]{Zuo96}, by replacing $X$ by some finite \'etale cover if necessary,  such a representation splits into a direct sum of one-dimensional representations. Hence $E_1$ is a direct sum of flat line bundles. Let $Q$ be one of them.   Then there is a   line bundle quotient $E\>L  $ with $L=Q^{-1}$. Moreover, since $L$ is also flat, it is numerically trivial.  This contradicts Proposition \ref{cri0}.

Now we study the general case. Let $\phi: X \dashrightarrow Y$ be the Iitaka fibration for $\omega_X$. Let $F$ be the closure of a general fiber of $\phi$. Then the Kodaira dimension of $F$ is $0$. Moreover,  $F$ has positive dimension as $\kappa (X)<\dim X$. By Proposition \ref{directsum}, we see that $E|_F$ is again strictly nef. Then from the first paragraph, the restriction of  $\det\, E$ on $F$  is not numerically trivial. Thus $\det\, E$ is not numerically trivial.
\eproof

\noindent Now we can conclude Theorem \ref{main5}.

 \begin{proof}[{Proof of Theorem \ref{main5}}]   The vector bundle $E$ can be decomposed as  $E=\ds E_i$ so that each $E_i$ is an indecomposable bundle. By Proposition  \ref{directsum},
  each $E_i$ is strictly nef.  Hence we have $\deg\, (E_i)>0$ by Lemma \ref{flat}. This implies that $E$ is ample (see  \cite[Theorem 1.3]{Har71} or \cite[Theorem 2.3]{Gie71}).
\end{proof}


\begin{rem}  From Proposition \ref{directsum} and Theorem \ref{main5}, we can obtain that if a  projective  variety $X$ contains the image of an elliptic curve (or a rational curve) which is not a point,
 then the determinant of every strictly nef vector bundle $E$ on $X$ is not numerically
 trivial. As a consequence, if $E$ is strictly nef over a projective variety $X$ with pseudo-effective
 $\omega^{-1}_X$, then $\det E$ is not numerically trivial. Indeed,  if $\omega_X$  is  not pseudo-effective, then   by \cite[Corollary~0.3]{BDPP13},  $X$ is covered by rational curves. If both $\omega_X$   and $\omega^{-1}_X$ are pseudo-effective, then $\omega_X$ is numerically trivial. Then the Kodaira dimension of $X$ is zero by  Beauville's decomposition theorem (see \cite[Th\'eor\`eme 1]{Bea83}), and we can apply Lemma \ref{flat} to conclude.
 \end{rem}

\section{Sections of Albanese morphisms}

\noindent In this section, we shall prove  Theorem \ref{thm-section2}. We will divide this section into two parts. In the first one, we will prove a theorem on periodic points for group actions on projective schemes. By using this theorem, we will conclude Theorem \ref{thm-section2}.

\subsection{Periodic points for linear actions of abelian groups}

The goal of this subsection is to prove the following theorem.

\begin{thm}
\label{thm-periodic-points} Let $G$ be a finitely generated abelian
group. Assume that $G$ acts on a projective scheme $Z$ such that
there is a $G$-equivariant ample line bundle $L$ on $Z$. Then the
action of $G$ on $Z$ has a periodic point.
\end{thm}

\noindent We recall that if $G$ acts on a projective scheme $Z$, then a line bundle $L$ on $Z$ is said to be $G$-equivariant if there is an isomorphism $g^*L\cong L$ for any element $g\in G$ which is compatible with the group structure of $G$. A (closed) point
$z\in Z$ is said to be a periodic point for an element $g \in G$ if $g^k.z=z$ for some positive integer $k$. A point $z$ is called a periodic point for the action of $G$ if there is a positive integer $k$ such that $g^k.z=z$ for all elements $g\in G$.

\begin{proof}[{Proof of Theorem \ref{thm-periodic-points}}]
By replacing $L$ with some positive power of it if necessary, we can
assume that $L$ is very ample.  Then there is a nature linear action of  $G$   on   $W=H^0(Z, L)$, which induces a natural action of $G$ on $\mathrm{Proj}\, W$.  Moreover, there is a
$G$-equivariant embedding $Z\to \mathrm{Proj}\, W$. The theorem is then
equivalent to the following proposition.
\end{proof}

\bproposition \label{prop-periodic-points} Let $G$ be a finitely
generated abelian group. Let  $V$ be a linear representation
of $G$. Assume  that $Z\subseteq \P(V)$ is a   closed subscheme which is
stable under the induced action of $G$ on $\P(V)$, where $\P(V)$ is the  projective  space of lines in $V$. Then the  action  of $G$ on $Z$ has a periodic point. \eproposition

\noindent  We will first prove Proposition \ref{prop-periodic-points}
in the case when   $G$ is generated by one element.

\begin{lemma}
\label{lem-rank-one} With the notations in Proposition
\ref{prop-periodic-points}, if we assume that $G$ is generated by one element, then the action of $G$ on $Z$ has a periodic point.
\end{lemma}

\begin{proof}
Assume that  $G$  is generated by  an element $g$. Then it is enough
to prove that some positive power of $g$ has  a periodic point in
$Z$. Hence during the proof, we will replace $g$ with some positive
power of it if necessary.    Let  $(x_0,...,x_m)$ be a coordinates
system  of $V$ such that the  vector with coordinates $(1,0,...,0)$
in $V$ is an eigenvector for $g$.   Let $[x_0:\cdots \cdot :x_m]$ be
the induced homogeneous coordinates system of $\P(V)$

We will prove the lemma by induction on the dimension $m$ of
$\P(V)$. If $m=1$, then $Z$ is either a finite set or  the whole
$\P(V)$. If $Z$  is a finite set, then all of its points are
periodic. If $Z=\P(V)$, then the point $[1:0]$ belongs to $Z$ and
is fixed by $g$. Hence the lemma is true in this case.

Assume that the lemma is true in dimensions smaller than $m\geqslant
2$. If the point $y$ with coordinates $[1:0:\cdots :  0]$ is in $Z$,
then it is a fixed-point for $g$ and we are done. Assume that $y$
does not belong to $Z$. Let $H\subseteq \mathbb{P}(V)$ be the
hyperplane of points whose $0$-th coordinates are  $0$.
 Then the rational projection $\phi: \P(V) \dashrightarrow H$ such that
 $$\phi ([x_0:  \cdots : x_m]) = [0: x_1:\cdots:x_m]$$   is a well-defined morphism on $Z$.
We also note that $\phi|_Z$ is proper, and hence the image
$Z'=\phi(Z)$ is  a closed subscheme of $H \cong \P^{m-1}$. Since
$(1,0,...,0)\in V$ is an eigenvector for $g$, the action of $G$ on
$\P(V)$ descends naturally to an action of $G$ on $H$ and the
rational projection $\phi$ is $G$-equivariant.  Thus $Z'$ is stable
under the action of $G$ on $H$.  By induction hypothesis, there is a
point $z'\in Z'$ which is a periodic point for the action of $G$. By
replacing $g$ with some positive power if necessary, we may assume
that $z'$ is a fixed-point. Let $L\subseteq \mathbb{P}(V)$ be the
line joining $y$ and $z'$. Then $L\cap Z$ is  non-empty  and stable
under the action of $G$.  Since we have assumed that $y\notin Z$,
the intersection $L\cap Z$ is a proper subset of $L$. Thus it is
a finite set.  In particular, every point in $Z\cap L$ is a periodic
point for $g$. This completes the proof.
\end{proof}

\noindent Now we can conclude  Proposition \ref{prop-periodic-points}.

\begin{proof}[{Proof of Proposition  \ref{prop-periodic-points}}]
We first note that it is enough to prove that some subgroup of $G$
of finite index has a periodic point. Hence during the proof we may
replace $G$ by some subgroup of finite index of it if necessary. In particular, we may assume that $G$ is torsion-free.
Moreover, without loss of generality, we may assume that the
representation $V$ of $G$ is faithful.  We will prove by induction
on the rank of $G$. If the rank is one, then the theorem follows
from Lemma \ref{lem-rank-one}.

Assume that the theorem holds  for ranks smaller than $k\geqslant
2$. Assume that $G$ has rank $k$. Let $\{g_1,...,g_k\}$ be a set of
generators of $G$. Let $F$ be the subgroup generated by $g_1$ and
let $H$ be the subgroup generated by $\{g_2,...,g_k\}$. Then by
Lemma \ref{lem-rank-one}, the action of $F$ on $Z$ has a periodic
point.   By replacing $g_1$ with some positive power of it and $G$
by some subgroup of finite index if necessary, we may assume that
the action of $F$ on $Z$ has a fixed-point.
In particular, the set $Z^{F}$  is not empty. We note that $Z^{F}$ is  also a  closed subscheme of
$\P^m$. Moreover, it is stable under the actions of $H$ since $G$  is
an abelian group. By induction hypothesis, the action of $H$ on
$Z^{F}$ has a periodic point $z$. Then $z$ is a periodic point for
the action of $G$ on $Z$. This completes the proof.
\end{proof}

\subsection{Proof of Theorem \ref{thm-section2} }

We will finish the proof of Theorem \ref{thm-section2}  in this subsection. We will need the following two lemmas.

\begin{lemma}
\label{lem-corres-fixed-codim-1-sub} Let $B$ be a smooth projective
variety and let $\pi: \widetilde{B}\to B$ be the universal cover
with Galois group $G=\pi_1(B)$.  Let $V$ be a linear representation of
$G$ and let $E$ be the corresponding  flat vector bundle over $B$.
Then there is a one-to-one correspondence between the set of
$G$-fixed-points $y\in \mathrm{Proj}\, V$ and the
set of codimension  one  flat subbundles $F \to E$.
\end{lemma}

\begin{proof}
There is a one-to-one correspondence between the set of
$G$-fixed-points $y$  and the set of codimension one
subrepresentations $W\subseteq V$. Moreover, the set of  codimension
one subrepresentations $W\subseteq V$ is one-to-one correspondent to
the set of codimension  one  flat subbundles $F \to E$.
\end{proof}

\noindent In the next lemma, we consider the following situation.  Let $Y$ be a  projective variety and let $H$ be a very ample line bundle on $Y$. Let $B$ be a smooth
projective variety with fundamental group $G=\pi_1(B)$. Let
$\widetilde{B}\to B$ be the universal cover.  Assume that there is
an action of $G$ on $Y$ such that $H$ is $G$-equivariant.  Let $G$
act on $Y\times \widetilde{B}$ diagonally and let $X$ be the
quotient $(Y\times \widetilde{B})/G$. Then there is a natural
fibration $f:X\to B$, and $H$ descends to a line bundle $L$ on
$X$. Moreover,  the natural linear action of $G$ on $V=H^0(Y, H)$ induces a flat vector bundle structure on $E=f_*L$.

\begin{lemma}
\label{lem-corres-section} With the notation in the paragraph above, we assume that there
is a $G$-fixed-point  $y\in Y$. Then on the one hand, $y$ induces a section $\sigma: B\to X$ of $f$. On the other hand, $y$ also induces a  short exact sequence of flat vector bundles $$0\to F \to E \to Q\to 0 $$   such that  $Q \cong \sigma^* L$.
\end{lemma}

\begin{proof}
Let $P=\mathrm{Proj}\, E$ and let $p:P\to B$ be the natural
projection. We note that there is a $G$-equivariant embedding $Y\to
\mathrm{Proj}\, V$, which induces  a closed embedding $X\to P$.
Moreover, we have $\sO_E(1)|_X \cong L$.

Since   $y\in Y \subseteq \mathrm{Proj}\, V$ is a  $G$-fixed-point,
it  corresponds to a codimension one flat subbundle $F\to E$ by
Lemma \ref{lem-corres-fixed-codim-1-sub}. Let $Q$ be the quotient
$E/F$. Then the quotient map $E\to Q$ induces a section $\mu: B\to
P$ of $p$ such that $\mu^*\sO_E(1)\cong Q$.  Moreover, for every
$b\in B$, $\mu(b)\in P_b$ corresponds exactly to $y\in
\mathrm{Proj}\, V$. Since $y\in Y$, the section $\mu$ factors
through a section   $\sigma: B\to X$ of $f$. Since $\sO_E(1)|_X =
L$, we have $Q \cong \sigma^* L$.  This completes the proof of the
lemma.
\end{proof}

\begin{rem} We note that the section $\sigma:B\to X$ in the lemma above depends only on the fixed-point $y$, and is independent of the choice of the line bundle $H$. Indeed, $\sigma(B)$ is the quotient $(\{y\}\times \widetilde{B}) / G$.
\end{rem}


\noindent  Now we can prove Theorem \ref{thm-section2}.

\begin{proof}[{Proof of Theorem \ref{thm-section2}}] As in
\cite[Corollary 4.16]{Cao16}, by replacing $X$ with some finite
\'etale cover, we may assume that the fibers of $f$ are simply
connected.  If $A$ is a point or if $f$ is an isomorphism, then
there is nothing to prove. Assume that $A$ has positive dimension
and that $f$ is not an isomorphism.  Let $L$ be an $f$-relatively
very ample divisor and let $E=f_*L$.  Assume that $\mathrm{rank}\,
E=m+1$ with $m\geqslant 0$ . There is an isogeny $p:A'\to A$ such
that $p^*(\mathrm{det}\, E)$ is divisible by $m+1$. Since we have
assumed that the fibers of $f$ are simply connected, the natural
morphism $X\times_A A' \to A'$ is still the Albanese morphism.
Hence, by replacing $X$ with $X\times_A A'$, we may assume that $
\mathrm{det}\, E = N^{m+1}$ for some line bundle $N$.  By replacing
$L$ with $L-f^*N$, we may then assume that $\mathrm{det}\, E  $ is
trivial.

Let $\pi: \widetilde{A} \to A $ be the universal cover with Galois
group $G=\pi_1(A)$.
Let $\widetilde{X}$ be the fiber product $X\times_A \widetilde{A}$ and
let $p:\widetilde{X} \to X$ be the natural morphism. By
\cite[Lemma 4.15]{Cao16}, $f_*(kL)$ is a numerically flat vector
bundle on $A$ for any positive integer $k$. In particular, $E$ is
numerically flat and hence is a flat vector bundle (see \cite[Lemma 6.5 and Corollary 6.6]{Ver04}
). Let $(V, \rho)$ be the corresponding representation of
$G$. Then there is a $G$-equivariant isomorphism $(\mathrm{Proj}\,
E) \times_A \widetilde{A} \cong (\mathrm{Proj}\, V)   \times
\widetilde{A}$, where the action of $G$ on $\mathrm{Proj}\, V$ is
the one induced by $\rho$. By \cite[Proposition 2.8]{Cao16}, there
is a $G$-stable subvariety $Y\subseteq  \mathrm{Proj}\, V$ such
that there is a $G$-equivariant isomorphism $\widetilde{X}\cong Y
\times \widetilde{A}$ which makes the following   diagram commute

\centerline{ \xymatrix{
Y \times \widetilde{A} \ar[d] \ar[r] &  (\mathrm{Proj}\, V)  \times \widetilde{A}  \ar[d]\\
X \ar[r]           &   \mathrm{Proj}\, E } } \noindent Let  $H=
\sO_{\mathrm{Proj}\, V}(1)|_Y$.  Then, on $\widetilde{X}$, we have  $p^*L\cong pr_1^* H $, where
$pr_1:\widetilde{X} \to Y$ is the natural projection induced by the
isomorphism $\widetilde{X} \cong Y \times \widetilde{A}$. After all,  we have $X\cong (Y\times \widetilde{A})/G$, and we are in the same situation as  in Lemma \ref{lem-corres-section}.

We note that $G$ is an abelian group. By Theorem
\ref{thm-periodic-points},  the action of $G$ on $Y$ has a periodic
point for $H$ is a $G$-equivariant ample line bundle. Hence there is
a subgroup $G'$ of $G$ of finite index such that the action of $G'$
on $Y$ has a fixed-point $y\in Y$.  The quotient  $\widetilde{X}/G'$
is a finite \'etale cover of $X$, and the natural morphism
$\widetilde{X}/G' \to \widetilde{A}/G'$ is the Albanese morphism for
we have assumed that the fibers of $f$ are simply connected. Hence,
by replacing $X$ with  $\widetilde{X}/G'$ if necessary, we may
assume that the action of $G$ on $Y$ has a fixed-point $y\in Y$.
This fixed-point induces a section $\sigma:A\to X$ of the Albanese
morphism $f$ by Lemma \ref{lem-corres-section}.

We note that the anticanonical bundle $\omega_Y^{-1}$ of $Y$ is nef for $\omega_X^{-1}$ is nef.
Moreover, $\omega_Y^{-1}$ is canonically $G$-equivariant.  Hence  $\omega_X^{-a} \otimes H$
is   a $G$-equivariant ample line bundle for any positive integer $a$. On $\widetilde{X}$, we  have  $$p^*(\omega_{X/A}^{-a} \otimes L) \cong
\omega_{\widetilde{X}/\widetilde{A}}^{-a} \otimes p^*L \cong  pr_1^*(\omega_Y^{-a} \otimes H).$$
Let $r_a$ be some large enough positive integer such that
$(\omega_Y^{-a} \otimes H)^{r_a}$ is very ample.   Then there is a natural linear action of $G$ on  $H^0(Y, (\omega_Y^{-a} \otimes H)^{r_a})$, which induces a flat vector bundle structure on  $E_a=  f_ *( \omega_{ {X}/ {A}}^{-a} \otimes L)^{r_a} $. By Lemma
\ref{lem-corres-section},   the $G$-fixed-point $y$ induces a short exact
sequence of flat vector bundles  $$0\to F_a\to E_a\to Q_a\to 0,$$
 such that $Q_a\cong \sigma^*  (\omega_{X/A}^{-a} \otimes L)^{r_a} $.  Since   $Q_a$ is flat, it is
numerically trivial. Thus $\sigma^*(\omega_{ {X}/ {A}}^{-a} \otimes L)$ is numerically
trivial. Since this is true for all positive integer $a$, we obtain
that $\sigma^*  \omega_{ {X}/ {A}} $ is numerically trivial. This completes the
proof of the theorem for $\omega_A$ is trivial.

\end{proof}


\section{Projective manifolds with strictly nef tangent bundles}

\noindent In this section, we will  prove Theorem \ref{main11}, Theorem
\ref{main2} and Theorem \ref{main3}. At first, we observe the
following fact.

\blemma \label{lem-product-non-snef} Let $X=Y\times Z$ be a variety of dimension $n$
which is a
product of two smooth projective varieties of positive dimensions.
 If  $\bigwedge^r T_X$ is strictly nef for
some $1\leqslant r \leqslant n$, then both $Y$ and $Z$ are uniruled. \elemma

\begin{proof}
We have $T_X\cong E \oplus F$, where $E$ and $F$ are the pullbacks
of the tangent bundles of $Y$ and $Z$ respectively. Then  we have $$\bigwedge^r T_X  \cong \bigoplus_{a+b=r} \bigwedge^a E
\otimes \bigwedge^b F.$$

Let $s$ be the dimension of $Y$. We will first show that  $r> s$.  Suppose by  contradiction that $r\leqslant s$. On the one hand, $\bigwedge^r E$ is a direct summand of $\bigwedge^r
T_X$. On the other hand, for any $y\in Y$, the restriction of $\bigwedge^r E$ on the fiber $X_y$ is trivial. Hence, $(\bigwedge^r
T_X)|_{X_y}$ cannot be strictly nef by Proposition \ref{directsum}.  This is a contradiction.

As a consequence, we obtain that  $N=\bigwedge^s E \otimes \bigwedge^{r-s} F$ is a direct
summand of $\bigwedge^r T_X$. In particular, $N$ is strictly nef by
Proposition \ref{directsum}.  Let  $z$ be a point in $Z$. Since   $(\bigwedge^{r-s} F)|_{X_z}$ is trivial and since $N|_{X_z}$ is still
strictly nef, we conclude that $(\bigwedge^s E)|_{X_z}$  is strictly
nef.  This implies that $\omega_Y^{-1}$ is strictly nef for $X_z\cong Y$.  Hence $Y$ is uniruled by
\cite[Corollary~2]{MM86}.
By symmetry, we can also show that $Z$ is uniruled.
\end{proof}

\noindent Now we are ready to prove Theorem \ref{main11}.

\begin{proof}[{Proof of Theorem \ref{main11}}]
We note that $\omega^{-1}_X$ is nef. We will first show that the augmented irregularity $\tilde{q}(X)$ is
zero. Assume the opposite. By replacing $X$ with some finite \'etale
cover if necessary, we may assume that the irregularity $q(X)$ is
equal to $\tilde{q}(X)>0$. Let $f:X\to A$ be the Albanese morphism.
Then $\mathrm{dim}\, A=q(X) >0$. Thanks to Theorem
\ref{thm-section2}, by  replacing $X$ with some finite \'etale cover if necessary,
we may assume that there is a section $\sigma:A\to X$ such that
$\sigma^* \omega_X$ is numerically trivial. On the one hand, we remark that $\mathrm{det}\, (\sigma^* \bigwedge^r T_X) \cong  \sigma^*\omega_X^{- t}$ is  numerically trivial, where $t$ is the binomial number $\binom{n-1}{r-1}$. One the other hand, since $\sigma^* \bigwedge^r T_X $ is  strictly nef, $\mathrm{det}\, (\sigma^* \bigwedge^r T_X)$ cannot be numerically trivial  by Lemma  \ref{flat}. We obtain a contradiction.

Therefore, we have $\tilde{q}(X)=0$. In particular, $X$ has finite
fundamental group. Let $\widetilde{X}\to X$ be the universal
cover. Then by \cite[Theorem 1.2]{CH17}, we have $\widetilde{X}\cong
Y\times F$ such that $\omega_Y$ is trivial and that $F$ is rationally
connected. By Lemma \ref{lem-product-non-snef}, $Y$ must be a point for  $\bigwedge^r T_ {\widetilde{X}} $ is strictly nef. Hence $\widetilde{X}$ is rationally connected and so is $X$.
\end{proof}

\noindent The following corollary is a direct consequence of Theorem \ref{main11}.

 \bcorollary\label{YXR} Let $f:Y\>X$ be
a smooth surjective morphism between projective manifolds. If $\omega_Y^{-1}$
is strictly nef, then $X$ is rationally connected. \ecorollary

\noindent  In the following, we will present two applications of Theorem
 \ref{main11} on characterizations of projective spaces and quadrics.  We will  prove Theorem \ref{main2} and Theorem \ref{main3} successively.

\begin{proof}[{Proof  of Theorem \ref{main2}}]
By Theorem \ref{main11} and the structure theorem
for  smooth projective varieties with nef tangent bundles (see \cite[Main Theorem]{DPS94}), we deduce that $X$ is a
Fano variety. For any rational curve $f:\P^1\>X$, the bundle $f^*T_X$  is ample by Proposition \ref{cri0}.  Moreover, since there is a non-zero morphism from $T_{\P^1}\cong \sO_{\P^1}(2)$ to $f^*T_X$,  we obtain that $\deg\, f^*\omega_X^{-1}  \geq n+1$.  Hence $X\cong \P^n$ by \cite[Corollary 0.3]{CMSB02}.
\end{proof}

\begin{proof}[{Proof of Theorem \ref{main3}}]

We know that  $X$ is rationally connected from Theorem \ref{main11}.  For a rational curve $f:\P^1 \> X$, we can write $$f^*T_X  \cong   \Bigg(\bigoplus_{a_i>0}\mathcal{O}_{\P^1}(a_i)\Bigg)  \bigoplus  \Bigg(\bigoplus_{b_j\leq0}\mathcal{O}_{\P^1}(b_j)  \Bigg).$$
Since $\bigwedge^2 T_X$ is  strictly nef, so is $f^*(\bigwedge^2 T_X)$ by Proposition \ref{cri0}. Hence   $\sharp\{b_j\}\le1.$ If  $\sharp\{b_j\}$ is $0,$ then $f^*T_X$ is ample, and   $\deg\, f^*\omega_X^{-1}\ge n+1$ by the same argument as in the proof of Theorem \ref{main2}.

If  $\sharp\{b_j\}$ is $1,$ then we can assume that  $$
f^*T_X \cong \Bigg(\bigoplus_{i=1}^{n-1}\mathcal{O}_{\P^1}(a_i)\Bigg) \bigoplus  \mathcal{O}_{\P^1}(-c)$$
 with $0<a_1\leq a_2 \leq \cdots \leq a_{n-1} $ and $c\ge 0$. Since $f^*(\bigwedge^2 T_X)$ is strictly nef, we must have  $a_1-c>0$.  Moreover, since there is a natural non-zero morphism  from $T_{\P^1}$ to  $f^*T_X$, there exists some $i$ such that $a_i\geq 2$. If $c>0$, then $a_1$ is at least $2$  and  we have  $$ \deg\, f^*\omega_X^{-1} =(a_1-c)+a_2+\cdots +a_{n-1}\geq
 1+2(n-2)=2n-3 \geq n. $$  If  $c=0$,  we also have $$ \deg\, f^*\omega_X^{-1} = a_1 +a_2+\cdots +a_{n-1}\geq
  (n-2) + 2= n. $$

After all, we always have $\deg\, f^*\omega_X^{-1}\ge n$. By \cite[Corollary~D]{DH}, $X$ is isomorphic to $\P^n$, or a quadric $\Q^n$, or a projective bundle  over some smooth curve. It remains to rule out the case of projective bundles. Assume by contradiction that $X\cong \mathrm{Proj}\,V$, where   $V$ is a vector bundle over a smooth  projective curve $B$.   We note that  $B$ is isomorphic to $\P^1$ for  $X$ is rationally connected.  We may assume that $V=\bigoplus_{i=1}^n \sO_{B}(d_i)$  with $0=d_1\le d_2 \le \cdots \le d_n$.  If  $\pi:X\to B$ is the natural projection, then we have
\begin{eqnarray*}
\omega_X^{-1} &= &(\pi^* \omega_{B}^{-1}) \otimes \omega_{X/B}^{-1} \\
 &\cong&  (\pi^*\sO_{B}(2)) \otimes (\sO_{V}(n) \otimes \pi^*( \det  V)^{-1}) \\
 &=&\sO_{V}(n ) \otimes \left(\pi^* \sO_{B}\left(2-\sum_{i=1}^n d_i\right)\right).
\end{eqnarray*}
  We note that the quotient morphism  $V\to \sO_{B}(d_1)$  induces a section $\sigma:B \to X$ of $\pi$ such that $\sigma^*\sO_{V}(1)\cong \sO_{B}(d_1) = \sO_{B}.$
   Thus we have $$\sigma^* \omega_X^{-1}  \cong \sO_{B}\left(2-\sum_{i=1}^n d_i\right).$$ This shows that  $\deg\,  \sigma^* \omega_{X}^{-1}  \le 2 <n$, which is  a contradiction.
\end{proof}

\begin{rem}
We note that if  a vector bundle $E$ is stictly nef, then  $\det\, E$ is not necessarily  strictly nef in general  (see \cite[Section 10 in  Chapter I]{Har70}). However,  inspired by Theorem \ref{main2} and Theorem \ref{main3}, we expect  that if $\bigwedge^r T_X$ is strictly nef for some $r>0$, then so is $-K_X$.  We then extend the conjecture of Campana and Peternell: if $\bigwedge^r T_X$ is strictly nef for some $r>0$, then $X$ is a Fano variety.
\end{rem}

\end{document}